\documentclass[12pt]{amsart}
\usepackage{graphicx, fullpage}
\usepackage{amsmath,amssymb,amsthm,amscd, bm}
\usepackage{fancyhdr}
\usepackage[mathscr]{eucal}
\usepackage{amsfonts}
\usepackage[T1]{fontenc}
\usepackage{xspace}
\usepackage{esint}
\theoremstyle{plain}
\newtheorem {thm}{Theorem}[section]
\newtheorem {lem}[thm]{Lemma}
\newtheorem {cor}[thm]{Corollary}
\newtheorem {defn}[thm]{Definition}
\newtheorem {prop}[thm]{Proposition}
\numberwithin{equation}{section}

\newcommand{\bC}{\mathbb C}
\newcommand{\bE}{\mathbb E}
\newcommand{\bK}{\mathbb K}

\newcommand{\bN}{\mathbb N}

\newcommand{\bR}{\mathbb R}

\newcommand{\var}{\varepsilon}

\newcommand{\GL}{\text{GL}}

\newcommand{\U}{\text{U}}

\newcommand{\vo}{\text{Vol\,}}

\begin{document}
\title{\bf Ellipsoids in pseudoconvex domains II}
\author{L\'aszl\'o Lempert}
\address{Department of  Mathematics,
Purdue University, 150N University Street, West Lafayette, IN
47907-2067, USA}

%\thispagestyle{empty}
%\end{titlepage}
\abstract
We consider families of pseudoconvex domains $X[s]\subset \bC^n$ parametrized by complex manifolds $S$,
and the maximal hermitian ellipsoids inscribed in each $X[s]$. The paper investigates how these maximal ellipsoids 
and their volume vary with $s$.
\endabstract
\maketitle
\section{Introduction}    % Section 1

Consider the space $\bE_0$ of hermitian ellipsoids\footnote{These are images of the unit ball in $\bC^n$ under invertible
linear transformations of $\bC^n$. In particular, they are centered at 0.}
 in complex Euclidean space $\bC^n$. In \cite{Le26} we studied ellipsoids 
$E\in\bE_0$ of largest volume inscribed in a given pseudoconvex domain $Y\subset\bC^n$,
$0\in Y$---a complex variant of a
problem that John dealt with in \cite{J48}. This ellipsoid of maximal volume will be called maximal ellipsoid (in $Y$). 
Under suitable assumptions we proved the existence and uniqueness of the maximal ellipsoid, and gave a 
characterization. For example, if $Y$ is bounded and strongly pseudoconvex, the maximal ellipsoid exists and is
unique. The current paper is about families of pseudoconvex open sets and how the maximal ellipsoids and their volume vary
in the family.

Let $S$ be a complex manifold, and fix a natural number $n$. By a family of pseudoconvex open sets (over $S$) we mean 
an open set $X\subset S\times\bC^n$ that contains $S\times\{0\}$, such that for each $s\in S$ the fiber
\begin{equation} %1.1
X[s]=\{x\in\bC^n:(s,x)\in X\}\subset\bC^n
\end{equation}
is pseudoconvex.

\begin{defn} %1.1
If $X$ above is a Stein manifold, we call it a pseudoconvex family of pseudoconvex open sets, or just a pseudoconvex
family, for short.
\end{defn}
\begin{thm} %1.2
Let $X\subset S\times\bC^n$ be a pseudoconvex family, and for each $s\in S$ let $v(s)$ denote the volume of the
maximal hermitian ellipsoid in $X[s]$ if this latter exists; in general, let
\[
v(s)=\sup\{\vo E:E\in\bE_0,\, E\subset X[s]\}\in(0,\infty].
\]
Then $-\log v:S\to[-\infty,\infty)$ is a plurisubharmonic function.
\end{thm}

A special case---when the $X[s]$ above are unit balls of norms on $\bC^n$---already arose in Rochberg's work on 
interpolation, see \cite[Theorems 3.1, 4.2]{Ro84}. After the proof of Proposition 4.3 we explain how Rochberg's assumptions
and ours are related.

Consider now an open $Y\subset\bC^n$, and for each $y\in Y$ let 
\[
w(y)=\sup\{\vo E:E\in\bE_0,\,E+y\subset Y\}\in(0,\infty].
\]
\begin{cor} %1.3
If $Y$ is pseudoconvex, then $-\log w$ is plurisubharmonic on $Y$.
\end{cor}

Indeed, this follows from Theorem 1.2 upon setting 
\[
S=Y,\quad X=\{(s,x): s,x+s\in Y\}\subset S\times\bC^n,
\]
since $v$ of the theorem and $w$ of the corollary agree.---It is worthwhile to compare Corollary 1.3 with one definition of
pseudoconvexity, namely, that $Y\ni y\mapsto-\log\text{ dist}(y,\partial Y)$ is plurisubharmonic. This says that if
instead of ellipsoids centered at $y\in Y$ we consider the volume $w'(y)$ of the maximal ball $\subset Y$ centered
at $y$, then $-\log w'$ is plurisubharmonic.

Theorem 1.2 is analogous to the Brunn--Minkowski inequality in its multiplicative formulation; the analogy is even 
closer with Berndtsson's theorems. By \cite[Theorem 1.4]{B98}, if the fibers of a pseudoconvex family
$X\subset S\times\bC^n$ are balanced\footnote{i.e., $x\in X[s]$ and $\zeta\in\bC$, $|\zeta|\le1$ imply $\zeta x\in X[s]$}
and $u(s)$ is the volume of the fiber $X[s]$ itself, then $-\log u$ is plurisubharmonic. Related results are due to
Cordero--Erausquin, Hamano, Maitani, Yamaguchi, and Berndtsson himself \cite{B98, B06, CE02, CE05, HY04, MY04}.
In particular, \cite[Theorem 1.1]{B06} provides other ways to associate numbers with domains in $\bC^n$ that,
in pseudoconvex families, vary in a plurisubharmonic way.

There are more results of similar flavor, both in complex and real geometry. Prior to the work of Berndtsson and 
Cordero--Erausquin, Yamaguchi \cite{Y89} discovered  that in certain families $X\subset\bC^n$, including
pseudoconvex families, the so called Robin constant of $X[s]$ is a logarithmically plurisubharmonic function of $s$.
Rashkovskii \cite[Proposition 4.2]{Ra16} proved that in certain families $L_t\subset\bC^n$, $t\in[0,1]$, of compact sets the 
Monge--Amp\`ere capacity of $L_t$ is a convex function of $t$. For Brunn--Minkowski--type property of a variety of
characteristics of convex domains in $\bR^n$, see e.g. the works of Colesanti and Salani \cite{C05, S05, S12}.

In light of Theorem 1.2 it is natural to ask whether in a pseudoconvex family $X\subset S\times\bC^n$ 
the maximal ellipsoids $E_s\in\bE_0$ in the fibers $X[s]$
themselves, and not only their volumes, have some plurisubharmonicity property. Do they constitute the fibers
of a pseudoconvex family $F\subset S\times\bC^n$? If this were so, Theorem 1.2 would be a consequence of
Berndtsson's theorem \cite[Theorem 1.4]{B98}. One has to be careful even with posing the question, though: the
maximal ellipsoids are not always unique, even if the fibers $X[s]$ are bounded. But suppose that each fiber is bounded
and strongly pseudoconvex. Then the maximal ellipsoids $E_s\in\bE_0$ inscribed in $X[s]$ are unique 
\cite[Theorem 1.2]{Le26}, form the fibers of a family 
$F=\bigcup_{s\in S}\big(\{s\}\times E_s\big)\subset S\times\bC^n$,
and the question whether $F$ has to be a pseudoconvex family is meaningful. However, the answer is in the
negative: if $n\ge 2$, $F$ need not be pseudoconvex. In fact, this failure is not just a property of maximal ellipsoids.
The same arises whenever one associates ellipsoids with domains in $\bC^n$ in a functorial way:

\begin{thm} %1.4
Fix $n=2,3,\dots$, and let $\Psi$ be a map from the space $\bK$ of strongly convex, smoothly bounded open 
neighborhoods of $0\in \bC^n$ to $\bE_0$. Suppose $\Psi$ is $\GL_n(\bC)$ equivariant:
\[
\Psi(AY)=A(\Psi Y)
\]
whenever $Y\in\bK$ and $A:\bC^n\to\bC^n$ is linear and invertible. Then there is a pseudoconvex family
$X\subset S\times\bC^n$ whose fibers are in $\bK$, but
\begin{equation*}
\bigcup_{s\in S}\big(\{s\}\times\Psi(X[s])\big)\subset S\times \bC^n
\end{equation*}
is not a pseudoconvex family.
\end{thm}

Equivariance means that the association $Y\mapsto\Psi Y$ depends only on the vector space structure of
$\bC^n$, and not on the choice of the coordinates.---In section 4 we will prove a slightly stronger result, with 
$\bK$ replaced by $\bK'\subset\bK$, the set of
balanced domains in $\bK$. 
%($Y\subset \bC^n$ is balanced if $y\in Y$ and $\zeta\in\bC$, $|\zeta|\le 1$ imply $\zeta y\in Y$).

\section{An $\bE_0$-valued Dirichlet problem} %section 2

This section prepares the proof of Theorem 1.2.
Fix an ellipsoid in $\bE_0$, say, the unit ball $B$ in $\bC^n$. Any $E\in\bE_0$ is of form $TB$, where $T$ is a 
positive self adjoint operator on $\bC^n$, uniquely determined by $E$. Positive self adjoint operators constitute an open
subset in the vector space of all self adjoint operators, and we use the association $E=TB\mapsto T$ to endow
$\bE_0$ with the structure of a smooth manifold. Let $\Delta\subset\bC$ denote the unit disc.

\begin{thm} %2.1
Consider a continuous map $f:\partial\Delta\to\bE_0$.\\
\phantom{aa} (a) There is a unique holomorphic $H:\Delta\to\GL_n(\bC)$ such that $H(0)$ is positive self adjoint and
\begin{equation} %2.1
F(s)=\begin{cases}H(s)B &\text{if } s\in\Delta\\
         f(s) &\text{if } s\in\partial\Delta\end{cases}
\end{equation}  
defines a continuous map $F:\overline\Delta\to\bE_0$.\\
\phantom{aaa}(b) Suppose $S\subset\bC$ is a neighborhood of $\overline\Delta$ and $X\subset S\times\bC^n$ 
is a pseudoconvex
family. If $f(s)\subset X[s]$ for $s\in\partial\Delta$, then $F$ of (2.1) satisfies $F(s)\subset X[s]$ for 
$s\in\overline\Delta$.
\end{thm}      

Part (b) is analogous to \cite[Theorem1.3]{Le26}, that concerns geodesic segments in the symmetric space
$\bE_0\approx \GL_n(\bC)/\U_n$: if the endpoints of a geodesic are inscribed in a pseudoconvex 
domain $Y\subset\bC^n$, then 
the same holds for all points of the geodesic.
\begin{proof}(a) If $f$ is somewhat smooth, the solution to our problem is immediately obtained from a certain
matrix factorization, related to Birkhoff's, see e.g. \cite[Theorem (8.1.1)]{PS88}. From here the case of a 
continuous $f$ can be obtained by
approximation. The statement also follows from a more general---and more difficult--- result of Coifman and Semmes
\cite[Section 17]{CS93}. Here we prove through a simple reduction to \cite[Theorem 3.1]{Le17}.

We can write $f(s)=T(s)B$, $s\in\partial\Delta$, with a continuous map $T:\partial\Delta\to\GL_n(\bC)$ whose values 
are positive self adjoint operators.  By \cite[Theorem 3.1]{Le17} there is a unique holomorphic map
$H:\Delta\to\GL_n(\bC)$ such that $H(0)$ is positive self adjoint and
\begin{equation} %2.2
P(s)=\begin{cases} H(s)H(s)^* &\text{if }s\in\Delta\\
               T(s)^2 &\text{if } s\in\partial\Delta\end{cases}
\end{equation}
defines a continuous $P:\overline\Delta\to GL_n(\bC)$. (\cite{Le17} provides a factorization of form $K(s)^*K(s)$
in the first line of (2.2), but applying such a factorization to $T(s)^{-2}$ we obtain the $H(s)=K(s)^{-1}$ of (2.2).)

Since $H=P^{1/2}(P^{-1/2}H)$, and the second factor here $=(HH^*)^{-1/2}H$ is unitary, 
$F(s)=H(s)B=P(s)^{1/2}B$
for $s\in\Delta$, and the continuity of $P$ implies the continuity of $F$. Thus $H$ solves the problem considered.

To see uniqueness, start with an $H$ as in the theorem. 
Defining $P$ by (2.2), and noting again that $H(s)B=P(s)^{1/2}B$ for
$s\in\Delta$, we see that the continuity of $F$ implies the continuity of $P$. Hence uniqueness in \cite[Theorem 3.1]{Le17}
implies uniqueness in our current problem.

(b) The proof is similar to that of \cite[Theorem 1.3]{Le26}. We start by recording that the solution $F$ in part (a) is stable:
If $f_j:\partial \Delta\to\bE_0$, $j\in\bN$, converge uniformly to $f:\partial\Delta\to\bE_0$ as $j\to\infty$, 
then the corresponding
$F_j:\overline\Delta\to\bE_0$ also converge uniformly to $F:\overline\Delta\to\bE_0$. This follows from the 
corresponding stability property of (2.2), \cite[Corollary 3.3]{Le17}.

We will prove (b) by the continuity method. Assume first that not only $f(s)$, but even its closure $\overline{f(s)}$ is 
inscribed in $X[s]$, $s\in\partial\Delta$. Choose $r>0$ so that each $X[s]$, $s\in\overline\Delta$, contains 
$r\overline B$, and construct a homotopy $f_t:\partial\Delta\to\bE_0$, $0\le t\le 1$
so that 
\[
[0,1]\times\partial\Delta\ni(t,s)\mapsto f_t(s)\in\bE_0
\]
is continuous, $f_0\equiv rB$, $f_1=f$, and $\overline{f_t(s)}\subset X[s]$ for all $t\in[0,1]$, $s\in\partial\Delta$. For each $t$ let
$H_t:\Delta\to\GL_n(\bC)$ be holomorphic such that $H_t(0)$ is positive self adjoint and
\[
F_t(s)=\begin{cases} H_t(s)B &\text{if } s\in\Delta\\
           f_t(s) &\text{if } s\in\partial\Delta\end{cases}
 \]
 defines a continuous $F_t:\overline\Delta\to\bE_0$. Our initial observation implies that $F_t(s)\in\bE_0$ depends 
 continuously on $t,s$. Let
 \[
 \Theta=\{t\in[0,1]: F_t(s)\subset X[s]\text{ for all } s\in\overline\Delta\}.
 \]
 For example, $0\in\Theta$, as $H_0\equiv rI$ fits the bill.
 
 To show that $\Theta\subset[0,1] $ is both open and closed, take a continuous plurisubharmonic exhaustion function $\phi:X\to\bR$. Let
 \[
 C=\bigcup\big\{\{s\}\times\overline{f_t(s)}:s\in\partial\Delta,\, t\in[0,1]\big\}\subset X\quad\text{and}\quad M=\sup _C \phi.
 \]
 Since $C$ is compact, $M<\infty$. Suppose $\theta\in\Theta$. With any $b\in B$, the function $\psi$ given by
 \[
 \psi(s)=\phi\big(s,H_\theta(s)b\big),\qquad s\in\Delta,
 \]
 is subharmonic. By the maximum principle
 \[
 \psi(s)\le\lim_{\rho\nearrow 1} \max_{\partial(\rho\Delta)} \psi\le
       \varlimsup_{\rho\nearrow 1}\max_{\sigma\in\partial(\rho\Delta)}\max_{F_\theta(\sigma)} \phi(\sigma,\cdot).
 \]
 As the last maximum here is a continuous function of $\sigma\in\overline\Delta$,
 \[
 \phi\big(s,H_\theta(s)b\big)=
 \psi(s)\le\max_{\sigma\in\partial\Delta}\max_{F_\theta(\sigma)}\phi(\sigma,\cdot)\le M,\qquad s\in\Delta
 \]
 follows. Since $b\in B$ was arbitrary,
 \[
 \bigcup_{s\in\Delta}\big(\{s\}\times F_\theta(s)\big)=\bigcup_{s\in\Delta}\big(\{s\}\times H_\theta(s)B\big)
 \subset\big\{(s,x)\in X: \phi(s,x)\le M\big\}=C_1,
 \]
 this latter being a compact subset of $X$. On the one hand this shows that if $t$ is sufficiently close to $\theta$, by
 continuity
 $F_t(s)\subset X[s]$ for all $s\in\overline\Delta$, and $\Theta$ is open. On the other, if $t$ is in the closure of $\Theta$,
then $F_t(s)\subset\{x\in\bC^n:(s,x)\in C_1\}\subset X[s]$, so that $t\in\Theta$ and $\Theta$ is
 closed.
 
 We conclude that $\Theta=[0,1]$. In particular $F(s)=F_1(s)\subset X[s]$ for $s\in\overline\Delta$.
 
 In the general case when for $s\in\partial\Delta$ only $f(s)$, rather than its closure, is known to be inscribed in $X[s]$,
 take $\lambda\in(0,1)$ and replace $f(s)$ by $\lambda f(s)\in\bE_0$. Applying what we just proved for $\lambda f$
 and letting $\lambda\to 1$ we obtain (b) in full generality.
\end{proof}

Theorem 2.1 remains true if $\bC^n$ is replaced by a general Hilbert space, because \cite{Le17}, on which the proof
was based, works in that generality. One can replace the disc as well, by any bordered Riemann surface, although then
$F$, the extension of the given $f$, can be  represented through holomorphic functions $H$ only locally. This
follows from Wu's work, \cite[Theorem 1.1]{W20}.

\section{The proof of Theorem 1.2} %section 3

We first investigate to what extent the maximal inscribed hermitian ellipsoid in an open
$Y\subset\bC^n$, and its volume, are stable under perturbation of $Y$. Assuming $0\in Y$, set
\begin{equation} %3.1
V(Y)=\sup\{\vo E:E\in\bE_0, E\subset Y\}\in (0,\infty].
\end{equation}
\begin{prop} %3.1
If $Y_1\subset Y_2\subset\dots\subset\bC^n$ is an increasing sequence of open sets containing $0$ and 
$Y=\bigcup_{j=1}^\infty Y_j$, then $V(Y)$ is the increasing limit of $V(Y_j)$.
\end{prop}
\begin{proof}
(3.1) implies that $V(Y_j)\le V(Y)$ form an increasing sequence, and so $\lim_jV(Y_j)\le V(Y)$. If $w<V(Y)$ there is an
$E\in\bE_0$ inscribed in $Y$ such that $w<\vo E$. At the price of a mild shrinking we can arrange that even
$\overline E\subset Y$. By compactness there is a $k\in\bN$ such that $\overline E\subset Y_k$, whence
\[
w<\vo E\le V(Y_k)\le\lim_{j} V(Y_j).
\]
This being true for all $w<V(Y)$, the claim follows.
\end{proof}

Next recall the notion of a holomorphic disc in a complex manifold $Z$. It is the range of a nonconstant holomorphic
map $\Delta\to Z$ ($\Delta$ is still the unit disc in $\bC$). According to \cite[Theorems 1.1, 1.2]{Le26}:
\begin{prop} %3.2
Suppose $Y\subset \bC^n$ is a bounded and pseudoconvex open set, $0\in Y$. If no holomorphic
disc in $\bC^n$ is contained in $\partial Y$, then there is a unique  $E\in\bE_0$ that is maximal in $Y$.
\end{prop}

In \cite{Le26} we only work with domains, that is, connected open sets, but connectedness is irrelevant, because the
maximal ellipsoid in a disconnected $Y$ is the same as in its component that contains $0$.---Recall that an open
subset $O$ of a complex manifold $Z$ is strongly pseudoconvex if each $y\in\partial O$ has a neighborhood $U$
with a strongly plurisubharmonic function $\phi\in C^2(U)$ such that $O\cap U=\{z\in U:\phi(z)<0\}$ and $d\phi$
does not vanish at $y$.

\begin{lem} %3.3
Consider a pseudoconvex family $X\subset S\times\bC^n$ with $X$ strongly pseudoconvex and the projection
$\overline X\to S$ proper. For every $s\in S$ the fiber $X[s]$ contains a unique maximal ellipsoid $E_s\in\bE_0$, and
$s\mapsto E_s$ defines a continuous map $S\to\bE_0$.
\end{lem}
\begin{proof}
The assumption implies that any holomorphic disc in $\overline X\subset S\times \bC^n$ is contained in $X$. Indeed,
suppose $h:\Delta\to S\times\bC^n$ is holomorphic, $h(\Delta)\subset \overline X$, and $h(\zeta)\in\partial X$
for some $\zeta\in\Delta$. Choose a neighborhood $U\subset S\times\bC^n$ of $h(\zeta)$ and 
a strongly plurisubharmonic $\phi\in C^2(U)$ such that 
$X\cap U=\{z\in U: \phi(z)<0\}$, and $d\phi$ does not vanish at $h(\zeta)$. Then $\phi\circ h$, defined near $\zeta$,
is subharmonic and attains its maximum, $0$, at $\zeta$. It follows that $\phi\circ h=0$ and 
$h^*\partial\bar\partial\phi=\partial\bar\partial (\phi\circ h)=0$ near $\zeta$, whence $h$ is constant (near $\zeta$, but
by analytic continuation, everywhere). Thus there is no holomorphic disc in $\overline X$ that intersects
$\partial X$.

In particular, for no $s$ does $\{s\}\times\partial X[s]\subset\partial X$ contain a holomorphic disc. By Proposition
3.2 therefore each fiber $X[s]$ contains a unique maximal ellipsoid $E_s\in\bE_0$.

Consider now a sequence $s_j\in S$ that converges to $s\in S$. With $B\subset\bC^n$ the unit ball we can write
$E_{s_j}=T_jB$, where $T_j$ are positive self adjoint operators on $\bC^n$. As $X$ is open, there is an $r\in(0,\infty)$
such that $rB\subset X[s_j]$ for all $j$; as the projection $\overline X\to S$ is proper, there is an $R\in(r,\infty)$
such that $X[s_j]\subset RB$. Therefore $\det T_j\ge r^n$, the operator norms $||T_j||\le R$, and the $T_j$ form a
relatively compact subset of the space of positive self adjoint operators.

Suppose first that the $T_j$ converge to an operator $T$. Since $\{s_j\}\times T_jB\subset X$, it follows that
$\{s\}\times TB\subset \overline X$. Now $\{s\}\times  TB$ is the union of holomorphic discs; as we saw, this implies that
$\{s\}\times TB\subset X$, i.e., $TB\subset X[s]$. Thus
\begin{equation} %3.2
\lim_j\vo T_jB=\vo TB\le\vo E_s.
\end{equation}
Also, if $w<\vo E_s$, there is a $\rho\in(0,1)$ such that $w<\vo \rho E_s$. For $j$ sufficiently large, 
$\rho\overline E_s\subset X[s_j]$, and so $w<\vo T_jB$. Hence $w\le\lim_j\vo T_jB$, and 
letting $w\to\vo E_s$ shows that in (3.2) equality holds. Therefore $TB\in\bE_0$ is 
maximal in $X[s]$. By uniqueness, $E_s=TB=\lim_j E_{s_j}$.

We are done if the $T_j$ are known to converge. Without this assumption we still obtain that any
subsequence of ${s_j}$ contains a further subsequence ${\sigma_k}$ along which $E_{\sigma_k}\to E_s$, by virtue of 
what we have already proved. This then implies that $E_{s_j}\to E_s$ along the entire sequence, and so
$s\mapsto E_s$ is continuous.
\end{proof}

\begin{proof}[Proof of Theorem 1.2]
Assume first that $X\subset S\times\bC^n$ is not only a pseudoconvex family, but $X$ is strongly pseudoconvex, and
the projection $\overline X\to S$ is proper. Lemma 3.3 implies that $-\log v$ is continuous; to show it is plurisubharmonic, we
need to prove the sub-mean-value property over discs in $S$. For this we can assume that the disc is the unit disc
$\Delta\subset\bC$ and $S\subset\bC$ is a neighborhood of $\overline\Delta$. As before, let $E_s\in\bE_0$ be the
maximal ellipsoid in $X[s]$. By Theorem 2.1 and Lemma 3.3 again there is a holomorphic $H:\Delta\to\GL_n(\bC)$ such that
\[
F(s)=\begin{cases} H(s)B &\text{if }s\in\Delta\\
      E_s &\text{if }s\in\partial\Delta\end{cases}
\]
defines a continuous $F:\overline\Delta\to\bE_0$; and $F(s)\subset X[s]$ for $s\in \overline\Delta$. Therefore
\(
v(0)\ge\vo F(0).
\)
As $\log\big(|\det H(s)|^2\vo B\big)=\log\vo F(s)$ is a harmonic function of $s\in\Delta$, 
the sub-mean-value property follows:
\[
\log v(0)\ge\frac1{2\pi}\int_0^{2\pi}\log\vo F(re^{it})\,dt\to\frac1{2\pi}\int_0^{2\pi}\log v(e^{it})\, dt
\quad\text{as } r\nearrow 1.
\]

Short of the extra assumptions on $X$, construct a smooth strongly plurisubharmonic exhaustion 
function $\phi:X\to\bR$. If $c\in\bR$ is a regular value of $\phi$ and $S_c=\{s\in S:\phi(s,0)<c\}$, then
\[
X_c=\{(s,x)\in S_c\times\bC^n: \phi(s,x)<c\}\subset S_c\times\bC^n
\]
is a pseudoconvex family, $X_c$ is strongly pseudoconvex, and projection of the relative closure of $X_c$ in
$S_c\times\bC^n$,
\[
\overline {X_c}=\{(s,x)\in S_c\times\bC^n:\phi(s,x)\le c\}\to S_c
\]
is proper. By what we have proved, if
\[
v_c(s)=\sup\{\vo E:E\in\bE_0,\, E\subset X_c[s]\}, \qquad s\in S_c,
\]
then $-\log v_c$ is plurisubharmonic. Therefore $-\log v$, by Proposition 3.1 the decreasing limit of $-\log v_c$ as
$c\to\infty$, is also plurisubharmonic.
\end{proof}

\section{Proof of Theorem 1.4} %section 4

Given $n=2,3,\dots$, let $\bK'$ be the set of smoothly bounded, strongly convex, balanced
domains $Y\subset\bC^n$. By
strong convexity we mean that the normal curvatures of $\partial Y$  are positive. Thus $Y\in\bK'$ are the unit balls of
smooth norms $||\,\,||$ on $\bC^n$ whose square is strongly convex: the real Hessian of $||\,\,||^2$, away from $0$, is positive definite. The following is a strengthening of Theorem 1.4:
\begin{thm}
If $\Psi:\bK'\to\bE_0$ is a $\GL_n(\bC)$ equivariant map: $\Psi(AY)=A\Psi(Y)$ whenever $A\in\GL_n(\bC)$ and
$Y\in\bK'$, then there is a pseudoconvex family $X\subset S\times\bC^n$ whose fibers are in $\bK'$ but
\begin{equation} %4.1
\bigcup_{s\in S}\big(\{s\}\times\Psi(X[s])\big)\subset S\times\bC^n
\end{equation}
is not a pseudoconvex family.
\end{thm}

The proof will be by contradiction, and we start by studying general maps $\Psi:\bK'\to\bE_0$ that associate with
pseudoconvex families $X\subset S\times\bC^n$ families (4.1) that are also pseudoconvex. If $\Psi$ has this
property, we will say that it preserves pseudoconvexity of families. We will write $\Psi X$ for the family (4.1).

For example, any constant map $\Psi$ preserves pseudoconvexity of  families. Less obviously, with notation (3.1),
$\Psi(Y)=V(Y)B$ also preserves pseudoconvexity of families. This
is just a reformulation of Theorem 1.2. Of course, these maps are not equivariant in the sense of Theorem 4.1; but
they become equivariant if we let the action of $\GL_n(\bC)$ on $\bE_0$ be induced by the trivial representation
of $\GL_n(\bC)$, respectively, the square of the determinantal representation, on $\bC^n$.

\begin{lem} %4.2
If $\Psi:\bK'\to\bE_0$ preserves pseudoconvexity of families, then it is monotone in the sense that 
$\Psi(Y_1)\subset\Psi(Y_2)$ whenever $Y_1,Y_2\in\bK'$, $\overline{Y_1}\subset Y_2$.
\end{lem}

We will need the following observation (certainly not original):
\begin{prop} %4.3
Let $S$ be a complex manifold and for each $s\in S$ let $||\,\,||_s$ be a norm on $\bC^m$. If
\begin{equation} %4.2
X=\{(s,x)\in S\times\bC^m:||x||_s<1\}
\end{equation} is a pseudoconvex family, then $log||y||_s$ is a plurisubharmonic function of
$s\in S$ for each $y\in \bC^m$.
\end{prop}
\begin{proof}
When $(s,x)\mapsto||x||_s$ is smooth, a proof can be easily put together from what is available in the literature, but 
there is no need to assume smoothness. It suffices to prove when $S\subset\bC^k$ is a ball
and $m=1$.
Then $X$ is a Hartogs domain, $X=\{(s,x):|x|<R(s)\}\subset S\times\bC$ with some function
$R:S\to(0,\infty]$. It is well known that pseudoconvexity forces $-\log R$ to be plurisubharmonic,
e.g., because $R$ is the pullback along the embedding $S\ni s\mapsto(s,0)\in X$ of $\delta_{(0,1)}$, distance to
$\partial X$ in the complex direction $(0,1)\in\bC^k\times\bC$ and, according to Lelong,
$-\log\delta_{(0,1)}$ is plurisubharmonic since $X$ is pseudoconvex \cite[Theorem 1]{L52} or \cite[Theorem 2.4.2]{L68}.
Hence $\log||y||_s=\log|y|-\log R(s)$ is plurisubharmonic.
\end{proof}

In this paper we will not need it, but it is easy to generalize Proposition 4.3: under its assumption, with any holomorphic
$h:S\to\bC^m$ the function $s\mapsto \log ||h(s)||_s$ is plurisubharmonic. Indeed, it suffices to check this on the
set where $h\neq0$, i.e., we can assume $h$ nowhere vanishes on $S$. But then consider the family of norms on $\bC$
given by $||\zeta||'_s=||\zeta h(s)||_s$, $\zeta\in\bC$, and the family 
\(
X'=\{(s,\zeta)\in S\times\bC:||\zeta||'_s<1\}.
\)
This is the preimage of $X$ under the holomorphic map 
\[
S\times\bC\ni(s,\zeta)\mapsto \big(s,\zeta h(s)\big)\in S\times\bC^m,
\]
hence pseudoconvex. Proposition 4.3 therefore implies that $||h(s)||_s=||1||'_s$ is a logarithmically plurisubharmonic 
function of $s$. In fact, the converse is true, too: If $S$ is Stein, the norms $||\,\,||_s$ are such that the function
$S\ni s\mapsto\log||h(s)||_s$ is plurisubharmonic whenever $h:S\to\bC^m$ is holomorphic, and 
(4.2) defines an open $X$, then this $X$ is pseudoconvex. The reason is that the assumptions imply that the
defining function $S\times\bC^m\ni (s,x)\mapsto \log||x||_s$ of $X$ is plurisubharmonic.---In sum, all this shows that when
$S$ is an open subset of $\bC$, (4.2) is a pseudoconvex family if and only if $||\,\,||_s$ is what interpolation
theory \cite[p. 356]{Ro84} calls a subinterpolation family of norms.

\begin{proof}[Proof of Lemma 4.2]

We will include $Y_1,Y_2$ as fibers in a pseudoconvex family $X\subset S\times\bC^n$ over a 
disc $S\subset\bC$. The construction of $X$ goes as follows.
 
For $j=1,2$ let $||\,\,||_j$ be the norm whose unit ball is $Y_j$; then on $\bC^n\setminus\{0\}$, $||\,\,||_1^2>||\,\,||_2^2$ are
smooth and strongly convex. Let $q=\max_{\overline {Y_1}}||\,\,||_2^2<1$, and choose $p>0$ so that
\begin{equation} %4.3
p||\,\,||_1^2+q<1\qquad\text{on }\overline {Y_2}.
\end{equation}
The functions $u_1=p||\,\,||_1^2+q\ge q$ and $u_2=||\,\,||_2^2$ satisfy, with some $\var>0$,
\[
u_1>u_2+2\var\quad\text{on }\overline Y_1,\qquad u_2>u_1+2\var\quad\text{on }\partial Y_2.
\]

A regularized maximum function $M=M_{\var,\var}:\bR^2\to\bR$ is a smooth, convex function, increasing in both
variables, $M(a+c,b+c)=M(a,b)+c$, and
\[
M(a,b)\ge\max(a,b), 
\]
with equality if $|a-b|\ge2\var$, see \cite[I.5.18]{D12}. It follows that $u=M(u_1,u_2):\bC^n\to[0,\infty)$ is convex, smooth 
on $\bC^n\setminus\{0\}$; $u(\zeta y)=u(y)$ if $\zeta\in\bC$ is unimodular;  $u=u_1=p+q<1$ on $\partial Y_1$ (cf. (4.3))
and  $u=u_2=1$ on $\partial Y_2$. In fact, $u$ is strongly convex on $\bC^n\setminus\{0\}$. This follows from
\[
u(t)-\delta |t|^2=M\big(u_1(t)-\delta |t|^2,u_2(t)-\delta |t|^2\big),
\]
because locally $\delta>0$ can be chosen so that $u_j(t)-\delta |t|^2$ are convex, whence so is $u(t)-\delta |t|^2$.
Finally, if $y\in\bC^n\setminus\{0\}$, the function $\bR\ni \lambda\mapsto u(\lambda y)$ is even and strongly 
convex, hence
$du(\lambda y)/d\lambda\vert_{\lambda=1}>0$. Thus $y$ is a regular point of $u$. Therefore $\{y:u(y)<c\}$ is a 
smoothly bounded, strongly convex, balanced domain for any $c\in(q,\infty)$. Note that $q=u(0)=\min u$.

Accordingly, we have a pseudoconvex family over $S=\{s\in\bC:|s|<2-q\}$,
\begin{equation} %4.4
X=\{(s,x)\in\bC\times\bC^n: |s|+u(x)<2\}\subset S\times\bC^n,
\end{equation}
whose fibers are in $\bK'$. If $x\in\partial Y_2 $ then $u(x)=u_2(x)=1$, which implies $x\in\partial(X[1])$. Since
$X[1]$ is convex, $X[1]=Y_2$ follows. Similarly, $X[2-p-q]=Y_1$.

By assumption the family $\Psi X$ is also pseudoconvex. Its fibers $(\Psi X)[s]$ are unit balls of certain hermitian
norms on $\bC^n$, that we will denote $h_s$. Clearly $h_s=h_{|s|}$. By Proposition 4.3 $\log h_s$ is a subharmonic
function of $s$. Since $2-p-q>1$, see (4.3), the maximum principle gives $h_1\le h_{2-p-q}$, or
\[
\Psi(Y_1)=(\Psi X)[2-p-q]\subset(\Psi X)[1]=\Psi(Y_2),\qquad\text{q.e.d.}
\]
\end{proof}

\begin{proof}[Proof of Theorem 4.1]
We will argue when $n=2$. 
Suppose a $\GL_2(\bC)$ invariant $\Psi:\bK'\to\bE_0$ preserves pseudoconvexity of families. 
Let $D=\{(\xi,\eta)\in\bC^2: |\xi|,|\eta|<1\}$ be the bidisc and consider the 
operator $A$ given by
\[
A(\xi,\eta)=\left(\frac{3\xi}4+\frac\eta5,\frac{3\xi}4-\frac\eta5\right),\qquad (\xi,\eta)\in\bC^2.
\]
Thus $\overline{AD}\subset D$. We can find $Y\in\bK'$ so close to $D$ that $\overline{AY}\subset Y$. By
averaging, we can even arrange that $Y$ is invariant under the transformations $(\xi,\eta)\mapsto(\eta,\xi)$ and
$(\xi,\eta)\mapsto$ $(e^{i\sigma}\xi,e^{i\tau}\eta)$, $\sigma,\tau\in\bR$. It follows from equivariance that $\Psi(Y)$ is 
also invariant, hence a ball, say, $rB$. Lemma 4.2 gives $A(rB)=A(\Psi(Y)=\Psi(AY)\subset\Psi(Y)=rB$, 
a contradiction, since for any $\xi\in\bC$ we have with Euclidean length
$||\,\,||$
\[
||A(\xi,0)||=\left\Vert\left(\frac{3\xi}4,\frac{3\xi}4\right)\right\Vert=\sqrt\frac{18}{16}|\xi|>||(\xi,0)||.
\] 
\end{proof}

In fact, the pseudoconvex family $X$ of Theorem 4.1 can be chosen independently of $\Psi$. One can use the domain
$Y\in\bK'$ in the proof of the theorem and $Y_1=AY$, $Y_2=Y$ to construct the pseudoconvex family
$X$ of (4.4). Yet $\Psi X$ cannot be 
pseudoconvex, for if it were, $\Psi(Y_1)=A(rB)$ would be contained in $\Psi(Y_2)=rB$.

\end{document}